\documentclass[12pt,reqno]{amsart}
\usepackage{amscd,amssymb,amsthm}
\usepackage{graphicx}
\usepackage[pdftex]{hyperref}

\usepackage{enumerate}
\theoremstyle{plain}
\newtheorem{theorem}{Theorem}[section]
\newtheorem{corollary}[theorem]{Corollary}
\newtheorem{lemma}[theorem]{Lemma}

\newtheorem{proposition}[theorem]{Proposition}

\theoremstyle{definition}
\newtheorem{remark}[theorem]{Remark}

\numberwithin{equation}{section}

\newcommand{\reel}{\mathbb{R}}
\newcommand{\nat}{\mathbb{N}}
\newcommand{\ent}{\mathbb{Z}}
\newcommand{\comp}{\mathbb{C}}

\newcommand{\ds}{\displaystyle}
\newcommand{\vf}{\varphi}
\newcommand{\eps}{\varepsilon}
\newcommand{\abs}[1]{\left\vert #1\right\vert }

\newcommand{\bg}{\medskip\goodbreak}
\newcommand{\bn}{\medskip\nobreak}

\newcommand{\vers}{{\,\longrightarrow\,}}

\let\oldtocsection=\tocsection

\let\oldtocsubsection=\tocsubsection

\let\oldtocsubsubsection=\tocsubsubsection

\renewcommand{\tocsection}[2]{\hspace{0em}\oldtocsection{#1}{#2}}
\renewcommand{\tocsubsection}[2]{\hspace{2em}\oldtocsubsection{#1}{#2}}
\renewcommand{\tocsubsubsection}[2]{\hspace{4em}\oldtocsubsubsection{#1}{#2}}

\title[Inequalities for finite trigonometric sums]
{{\large Inequalities for finite trigonometric sums,}\\
	{\large An interplay:}\\
		{\large with some series related to Harmonic numbers}}
\author[Omran Kouba]{Omran Kouba$^\dag$}
\address{Department of Mathematics \\
	Higher Institute for Applied Sciences and Technology\\
	P.O. Box 31983, Damascus, Syria.}
\email{omran\_kouba@hiast.edu.sy}
\keywords{Bernoulli  polynomials, Bernoulli numbers,
	harmonic numbers, asymptotic expansion, sum of cosecants, sum of cotangents}
\subjclass[2010]{11B68, 11B83, 26D05, 26D15, 41A17.}
\thanks{$^\dag$ Department of Mathematics, Higher Institute for Applied Sciences and Technology.}

\begin{document}
\begin{abstract}
An interplay between the sum of certain series related to Harmonic numbers and certain finite trigonometric
sums is investigated. This allows us to express the sum of these series in terms 
of the considered trigonometric sums, and permits us to find sharp inequalities bounding these trigonometric sums. In particular,
this answers positively an open problem of H. Chen (2010).
\end{abstract}

\maketitle
\section{Introduction}\label{sec1}
\bg

Many identities that evaluate trigonometric sums in closed form  can be  found in the literature. For example, in a solution to a problem in SIAM Review \cite[p.157]{klam}, M.
Fisher  shows that
\[
\sum_{k=1}^{p-1}\sec^2\left(\frac{k\pi}{2p}\right)
=\frac{2}{3}\left(p^2-1\right),\quad
\sum_{k=1}^{p-1}\sec^4\left(\frac{k\pi}{2p}\right)
=\frac{4}{45}\left(2p^4+5p^2-7\right).
\]

General results giving closed forms for the power sums
secants 
$\sum_{k=1}^{p-1}\sec^{2n}(\frac{k\pi}{2p})$ and
${\sum_{k=1}^{p}\sec^{2n}(\frac{k\pi}{2p+1}})$, for many values of
the positive integer $n$, can be found in
\cite{chen} and \cite{grab}.
Also, in \cite{kou2} the author proves that \[\sum_{k=1}^{p}\sec\left(\frac{2k\pi}{2p+1}\right)
=\begin{cases}
\phantom{-}p&\text{if $p$ is even,}\\
-p-1& \text{if $p$ is odd.}
\end{cases}
\]

However, while there are many cases where closed 
forms for finite trigonometric sums can be obtained it seems that there are no such formul\ae\, for the sums we are interested in.

In this paper we study the trigonometric sums $I_p$ and $J_p$
defined for positive integers $p$ by the formul\ae:
\begin{align}
	I_p&=\sum_{k=1}^{p-1}\frac{1}{\sin(k\pi /p)}=\sum_{k=1}^{p-1}\csc\left(\frac{k\pi}{p}\right)\label{E:I}\\
	J_p&=\sum_{k=1}^{p-1}k\cot\left(\frac{k\pi}{p}\right)\label{E:J}
\end{align}
with empty sums interpreted as $0$. 
\bg
To the author's knowledge there is no known closed  form for $I_p$, and the same can be said
about the sum $J_p$. Therefore, we will look for  asymptotic expansions for these sums and will give some tight inequalities that bound $I_p$ and $J_p$. This investigation complements the work of H. Chen  in \cite[Chapter 7.]{chen2} where it was asked, as an open problem, whether the inequality
\[
I_p\le \frac{2p}{\pi}(\ln p+\gamma -\ln(\pi/2))
\]
holds true for $p\ge3$, (here
 $\gamma$ is the so called Euler-Mascheroni constant.) \bg
 In fact, it will be proved that
 for every positive integer $p$ and every nonnegative integer $n$, we have
 \begin{align*}
 	I_p&<\frac{2p}{\pi}(\ln p+\gamma-\ln(\pi/2))+
 	\sum_{k=1}^{2n}(-1)^{k}\frac{(2^{2k}-2)b_{2k}^2}{k\cdot(2k)!}\left(\frac{\pi}{p}\right)^{2k-1},\\
 	\noalign{\noindent\text{and}}
 	I_p&>\frac{2p}{\pi}(\ln p+\gamma-\ln(\pi/2))+
 	\sum_{k=1}^{2n+1}(-1)^{k}\frac{(2^{2k}-2)b_{2k}^2}{k\cdot(2k)!}\left(\frac{\pi}{p}\right)^{2k-1}.
 \end{align*}
where the $b_{2k}$'s are Bernoulli numbers (see Corollary \ref{cor94}. The corresponding inequalities for $J_p$ are also proved
(see Corollary \ref{cor97}.)
 
Harmonic numbers play an important role in this investigation. Recall that the $n$th harmonic number $H_n$ is defined by $H_n=\sum_{k=1}^n1/k$ (with the convention
$H_0=0$). In this work, a link between our trigonometric sums $I_p$ and $J_p$ and the 
sum of several series related to harmonic numbers is uncovered. Indeed, the well-known fact that $H_n=\ln n+\gamma+\frac{1}{2n}+\mathcal{O}\left(\frac{1}{n^2}\right)$ proves the convergence of the numerical series,
\begin{align*}
C_p&=\sum_{n=1}^\infty\left(H_{pn}-\ln(pn)-\gamma-\frac{1}{2pn}\right), \\
D_p&=\sum_{n=1}^\infty(-1)^{n-1}\left(H_{pn}-\ln(pn)-\gamma\right),\\
E_p&=
	\sum_{n=0}^\infty(-1)^{n}(H_{p(n+1)}-H_{pn}),
\end{align*}
for every positive integer $p$. 

An interplay between the considered trigonometric sums and the sum of these series will allow us to prove sharp inequalities for  $I_p$ and $J_p$, and to find the expression of the sums $C_p$, $D_p$ and $E_p$ in terms of $I_p$ and $J_p$.

The main tool will be the following formulation \cite[Corollary 8.2]{kou3} of the Euler-Maclaurin summation formula:

\begin{theorem}\label{cor61}
Consider a positive integer $m$, and a function $f$ that has a  continuous $(2m-1)^{\text{st}}$ derivative on $[0,1]$. If
 $f^{(2m-1)}$ is {\normalfont \text{decreasing}}, then
\[
\int_0^1f(t)\,dt=\frac{f(1)+f(0)}{2}
-\sum_{k=1}^{m-1}\frac{b_{2k}}{(2k)!}\,\delta f^{(2k-1)}+(-1)^{m+1}R_m
\]
with
\[
R_m=\int_0^{1/2}\frac{\abs{B_{2m-1}(t)}}{(2m-1)!}\,\left(f^{(2m-1)}(t)-f^{(2m-1)}(1-t)\right)\,dt
\]
and
\[
0\leq R_m\leq\frac{6}{(2\pi)^{2m}}\left(f^{(2m-1)}(0)-f^{(2m-1)}(1)\right).
\]
where the $b_{2k}$'s are Bernoulli numbers, $B_{2m-1}$ is the Bernoulli polynomial of degree $2m-1$, and the notation $\delta g$ for a function $g:[0,1]\to\comp$ means $g(1)-g(0)$.
\end{theorem}
For more information on the Euler-Maclaurin formula, Bernoulli polynomials and Bernoulli numbers the reader may refer to
\cite{abr, grad, hen,kou3,olv} and the references therein.
This paper is organized as follows. In section 2 we find the asymptotic expansions of $C_p$ and $D_p$ for large $p$. In section 3, the trigonometric sums $I_p$ and $J_p$ are studied.

\bg

\section{The sum of certain series related to harmonic numbers}\label{sec8}
\bn
In the next lemma, the  asymptotic expansion of $(H_n)_{n\in\nat}$ is presented.
It can be found implicitly in  \cite[Chapter 9]{knuth2} we present a proof for the convenience of the reader.
\bg
\begin{lemma}\label{pr81}
	For every positive integer $n$ and nonnegative integer $m$, we have
	\[
	H_n=\ln n+\gamma+\frac{1}{2n}-\sum_{k=1}^{m-1}\frac{b_{2k}}{2k}\cdot\frac{1}{n^{2k}}+(-1)^m R_{n,m},
	\]
	with 
	\[
	R_{n,m}=\int_0^{1/2}\abs{B_{2m-1}(t)}\,
	\sum_{j=n}^\infty\left(\frac{1}{(j+t)^{2m}}-\frac{1}{(j+1-t)^{2m}}
	\right)\,dt
	\]
	Moreover,  $ 0<R_{n,m}<\dfrac{\abs{b_{2m}}}{2m\cdot n^{2m}}$.
\end{lemma}
\begin{proof}
	Note that for $j\geq1$ we have
	\[
	\frac{1}{j}-\ln\left(1+\frac{1}{j}\right)=\int_0^1\left(\frac1j-\frac{1}{j+t}\right)\,dt=\int_0^1\frac{t}{j(j+t)}\,dt
	\]
	Adding these equalities as $j$ varies from $1$ to $n-1$ we conclude that
	\[
	H_n-\ln n-\frac{1}{n}=\int_0^1\left(\sum_{j=1}^{n-1}\frac{t}{j(j+t)}\right)\,dt.
	\]
	Thus, letting $n$ tend to $\infty$, and using  the Monotone Convergence Theorem, we conclude
	\[
	\gamma=\int_0^1\left(\sum_{j=1}^{\infty}\frac{t}{j(j+t)}\right)\,dt.
	\]
	It follows that
	\begin{equation*}
	\gamma+\ln n-H_n+\frac{1}{n}=\int_0^1\left(\sum_{j=n}^\infty\frac{t}{j(j+t)}\right)\,dt.
	\end{equation*}
	So, let us consider the function $f_n:[0,1]\vers\reel$ defined by
	\[
	f_n(t)=\sum_{j=n}^\infty\frac{t}{j(j+t)}
	\]
	Note that $f_n(0)=0$, $f_n(1)=1/n$, and that $f_n$ is infinitely continuously derivable with
	\[
	\frac{f_n^{(k)}(t)}{k!}=(-1)^{k+1}\sum_{j=n}^\infty\frac{1}{(j+t)^{k+1}},\quad\text{for $k\geq1$.}
	\]
	In particular, 
	\[
	\frac{f_n^{(2k-1)}(t)}{(2k-1)!}=\sum_{j=n}^\infty\frac{1}{(j+t)^{2k}},\quad\text{for $k\geq1$.}
	\]
	So, $f_n^{(2m-1)}$ is decreasing on the interval $[0,1]$, and
	\[
	\frac{\delta f_n^{(2k-1)}}{(2k-1)!} =\sum_{j=n}^\infty\frac{1}{(j+1)^{2k}}
	-\sum_{j=n}^\infty\frac{1}{j^{2k}}=-\frac{1}{n^{2k}}
	\]
	Applying Theorem \ref{cor61} to $f_n$, and using the above data, we get
	\[
	\gamma+\ln n-H_n+\frac{1}{2n}=
	\sum_{k=1}^{m-1}\frac{b_{2k}}{2k\,n^{2k}}+(-1)^{m+1}R_{n,m}
	\]
	with
	\[
	R_{n,m}=\int_0^{1/2}\abs{B_{2m-1}(t)}\,
	\sum_{j=n}^\infty\left(\frac{1}{(j+t)^{2m}}-\frac{1}{(j+1-t)^{2m}}
	\right)\,dt
	\]
	and
	\[
	0< R_{n,m}<\frac{6\cdot(2m-1)!}{(2\pi)^{2m}n^{2m}}.
	\]
	The important estimate is the lower bound, \textit{i.e.} $R_{n,m}>0$. In fact, considering separately the cases $m$ odd and $m$ even, we obtain, for every nonnegative integer $m'$:

	\begin{align*}
	H_n&<\ln n+\gamma+\frac{1}{2n}-\sum_{k=1}^{2m'}\frac{b_{2k}}{2k}\cdot\frac{1}{n^{2k}},\\
	\noalign{\noindent\text{and}}
	H_n&>\ln n+\gamma+\frac{1}{2n}-\sum_{k=1}^{2m'+1}\frac{b_{2k}}{2k}\cdot\frac{1}{n^{2k}}.\\
	\end{align*}
	This yields the following more precise estimate for the error term:
	\begin{equation*} 
	0<(-1)^{m}\left(H_n-\ln n-\gamma-\frac{1}{2n}+
	\sum_{k=1}^{m-1}\frac{b_{2k}}{2k\cdot n^{2k}} \right)<\frac{\abs{b_{2m}}}{2m\cdot n^{2m}}
	\end{equation*}
	which is valid for every positive integer $m$. 
\end{proof}
\bg

Now, consider the two sequences $(c_n)_{n\geq1}$ and $(d_n)_{n\geq1}$ defined by
\[
c_n=H_n-\ln n-\gamma-\frac{1}{2n}\qquad\text{and}\qquad d_n=H_n-\ln n-\gamma
\]
For a positive integer $p$, we know according to Lemma~\ref{pr81} that $c_{pn}=\mathcal{O}\left(\frac{1}{n^2}\right)$, it follows that
the series $\sum_{n=1}^\infty c_{pn}$ is convergent. Similarly, since $d_{pn}=c_{pn}+\frac{1}{2pn}$ and the series $\sum_{n=1}^\infty(-1)^{n-1}/n$ is convergent, we conclude that $\sum_{n=1}^\infty(-1)^{n-1} d_{pn}$ is also convergent.
In what follows we aim to find asymptotic expansions, (for large $p$,) of the following sums:
\begin{align}
C_p&=\sum_{n=1}^\infty c_{pn}=\sum_{n=1}^\infty\left(H_{pn}-\ln(pn)-\gamma-\frac{1}{2pn}\right)\label{E:Cp}\\
D_p&=\sum_{n=1}^\infty (-1)^{n-1}d_{pn} =\sum_{n=1}^\infty(-1)^{n-1}\left(H_{pn}-\ln(pn)-\gamma\right)\label{E:Dp}\\
	E_p &=\sum_{n=0}^\infty(-1)^{n}(H_{p(n+1)}-H_{pn}).\label{E:Ep}
\end{align}

\begin{proposition}\label{pr82}
	If $p$ and $m$ are positive integers and $C_p$ is defined by \eqref{E:Cp}, then
	\[
	C_p=-\sum_{k=1}^{m-1}\frac{b_{2k}\zeta(2k)}{2k\cdot p^{2k}}
	+(-1)^m\frac{\zeta(2m)}{2m\cdot p^{2m}}\eps_{p,m},\quad\text{with $0<\eps_{p,m}<\abs{b_{2m}}$},
	\]
	where $\zeta$ is the well-known Riemann zeta function.
\end{proposition}
\begin{proof} Indeed, we conclude from Lemma \ref{pr81} that 
	\[
	H_{pn}-\ln(pn)-\gamma-\frac{1}{2pn}=-\sum_{k=1}^{m-1}\frac{b_{2k}}{2k\cdot p^{2k}}\cdot\frac{1}{n^{2k}}
	+\frac{(-1)^m}{2m\cdot p^{2m}}\cdot\frac{r_{pn,m}}{n^{2m}}.
	\]
	with $0<r_{pn,m}\leq\abs{b_{2m}}$. It follows that
	\[C_p=-\sum_{k=1}^{m-1}\frac{b_{2k}}{2k\,p^{2k}}\cdot\left(\sum_{n=1}^\infty
	\frac{1}{n^{2k}}\right)+\frac{(-1)^m}{2m\cdot p^{2m}}\cdot\tilde{r}_{p,m}.
	\]
	where $\tilde{r}_{p,m}=\sum_{n=1}^\infty\frac{r_{pn,m}}{n^{2m}}$. 
	\bg
	Hence,
	\[
	0<\tilde{r}_{p,m}=\sum_{n=1}^\infty
	\frac{r_{pn,m}}{n^{2m}}< \abs{b_{2m}} \,
	\sum_{n=1}^\infty
	\frac{1}{n^{2m}}=\abs{b_{2m}}\zeta(2m)
	\]
	and the desired conclusion follows with $\eps_{p,m}=\tilde{r}_{p,m}/\zeta(2m)$.
\end{proof}
\bg
For example, taking $m=3$, we obtain
\[
\sum_{n=1}^\infty\left(H_{pn}-\ln(pn)-\gamma-\frac{1}{2pn}\right)
=-\frac{\pi^2}{72p^2}+\frac{\pi^4}{10800p^4}+\mathcal{O}\left(\frac{1}{p^6}\right).
\]

\bg
In the next proposition we have the analogous result corresponding to $D_p$.\bg

\begin{proposition}\label{pr83}
	If $p$ and $m$ are positive integers and $D_p$ is defined by \eqref{E:Dp}, then
	\[
	D_p=\frac{\ln 2}{2p}
	-\sum_{k=1}^{m-1} \frac{b_{2k}\eta(2k)}{2k \cdot p^{2k}}
	+(-1)^m\frac{\eta(2m)}{2m\cdot p^{2m}}\eps'_{p,m},\quad\text{with $0<\eps'_{p,m}<\abs{b_{2m}}$,}
	\]
	where $\eta$ is the Dirichlet eta function \cite{wis}. 
\end{proposition}
\begin{proof}
	Indeed, let us define $a_{n,m}$ by the formula
	\[
	a_{n,m}=H_n-\ln n-\gamma-\frac{1}{2n}+\sum_{k=1}^{m-1}\frac{b_{2k}}{2k\cdot n^{2k}}
	\]
	with empty sum equal to 0. We have shown in the proof of Lemma \ref{pr81} that
	\[
	(-1)^ma_{n,m}=\int_0^{1/2}\abs{B_{2m-1}(t)}g_{n,m}(t)\,dt
	\]
	where $g_{n,m}$ is the positive decreasing function on $[0,1/2]$ defined by
	\[
	g_{n,m}(t)=\sum_{j=n}^\infty\left(\frac{1}{(j+t)^{2m}}-\frac{1}{(j+1-t)^{2m}}\right).
	\]
	Now, for every $t\in[0,1/2]$ the sequence $(g_{np,m}(t))_{n\geq1}$ is positive and decreasing to $0$. So, using the alternating series criterion
	\cite[Theorem~7.8, and Corollary~7.9]{aman} we see that, for every $N\geq1$ and $t\in[0,1/2]$,
	\[
	\abs{\sum_{n=N}^\infty(-1)^{n-1}g_{np,m}(t)}\leq g_{Np,m}(t)\leq g_{Np,m}(0)=\frac{1}{(Np)^{2m}}.
	\]
	This proves the uniform convergence on $[0,1/2]$ of the series 
	\[G_{p,m}(t)=\sum_{n=1}^\infty(-1)^{n-1}g_{np,m}(t).
	\]
	Consequently
	\[
	(-1)^m\sum_{n=1}^\infty(-1)^{n-1}a_{pn,m}=\int_0^{1/2}\abs{B_{2m-1}(t)}G_{p,m}(t)\,dt.
	\]
	Now using the properties of alternating series, we see that for $t\in(0,1/2)$ we have
	\[
	0<G_{p,m}(t)<g_{p,m}(t)<g_{p,m}(0)=\sum_{j=p}^\infty\left(\frac{1}{j^{2m}}-\frac{1}{(j+1)^{2m}}\right)=\frac{1}{p^{2m}}
	\]
	Thus, 
	\[
	\sum_{n=1}^\infty(-1)^{n-1}a_{pn,m}=\frac{(-1)^m}{p^{2m}}\rho_{p,m}
	\]
	with $0<\rho_{p,m}<\int_0^{1/2}\abs{B_{2m-1}(t)}\,dt$.
	\bg
	On the other hand we have
	\begin{align*}
	\sum_{n=1}^\infty(-1)^{n-1}a_{pn,m}&=D_p
	-\frac{1}{2p} \sum_{n=1}^\infty\frac{(-1)^{n-1}}{n}+\sum_{k=1}^{m-1}\frac{b_{2k}}{2k\,p^{2k}}\sum_{n=1}^\infty\frac{(-1)^{n-1}}{n^{2k}}\\
	&=D_p-\frac{\ln 2}{2p} +\sum_{k=1}^{m-1}\frac{b_{2k}\eta(2k)}{2k\cdot p^{2k}}.
	\end{align*}
	Thus
	\[
	D_p=\frac{\ln 2}{2p}-\sum_{k=1}^{m-1}\frac{b_{2k}\eta(2k)}{2k\cdot p^{2k}}
	+\frac{(-1)^m}{p^{2m}}\rho_{p,m}
	\]
	Now,  the important estimate for $\rho_{p,m}$ is the lower bound, \textit{i.e.} $\rho_{p,m}>0$. In fact, considering separately the cases $m$ odd and $m$ even, we obtain, for every nonnegative integer $m'$:
	\begin{align*}
	D_p&<\frac{\ln 2}{2p}-\sum_{k=1}^{2m'}\frac{b_{2k}\eta(2k)}{2k\cdot p^{2k}},\\
	\noalign{\text{and}}
	D_p&>\frac{\ln 2}{2p}-\sum_{k=1}^{2m'+1}\frac{b_{2k}\eta(2k)}{2k\cdot p^{2k}}.\\
	\end{align*}
	This yields the following more precise estimate for the error term:
	\[
	0<(-1)^{m}\left(D_p-\frac{\ln 2}{2p}+\sum_{k=1}^{m-1}\frac{b_{2k}\eta(2k)}{2k\,p^{2k}}
	\right)<\frac{\abs{b_{2m}}\eta(2m)}{2m\cdot p^{2m}},
	\]
	and the desired conclusion follows.
\end{proof}
\bg

The case of $E_p$ which is the sum of another alternating series 
\eqref{E:Ep} is discussed in the next lemma where 
it is shown that $E_p$ can be easily expressed in terms of $D_p$.\bg
\begin{lemma}\label{lm84}
	For a positive integer $p$, we have
	\begin{equation*}
	E_p =\ln p+\gamma-\ln\left(\frac{\pi}{2}\right)+2D_p,
	\end{equation*}
	where $D_p$ is the sum defined by  \eqref{E:Dp}.
\end{lemma}
\begin{proof}
	Indeed
	\begin{align*}
	2D_p&=d_p+\sum_{n=2}^\infty(-1)^{n-1}d_{pn}+\sum_{n=1}^\infty(-1)^{n-1}d_{pn}\\
	&=d_p+\sum_{n=1}^\infty(-1)^{n}d_{p(n+1)}+\sum_{n=1}^\infty(-1)^{n-1}d_{pn}\\
	&=d_p+\sum_{n=1}^\infty(-1)^{n-1}(d_{pn}-d_{p(n+1)})\\
	&=d_p+\sum_{n=1}^\infty(-1)^{n}(H_{p(n+1)}-H_{pn})+\sum_{n=1}^\infty(-1)^{n-1}\ln\left(\frac{n+1}{n}\right)\\
	&=-\ln p-\gamma+\sum_{n=0}^\infty(-1)^{n}(H_{p(n+1)}-H_{pn})+\sum_{n=1}^\infty(-1)^{n-1}\ln\left(\frac{n+1}{n}\right)\\
	\end{align*}
	Using Wallis formula for $\pi$~ \cite[Formula~0.262]{grad},  we have
	\begin{align*}
	\sum_{n=1}^\infty(-1)^{n-1}\ln\left(\frac{n+1}{n}\right)&=
	\sum_{n=1}^\infty\ln\left(\frac{2n}{2n-1}\cdot\frac{2n}{2n+1}\right)\\
	&=-\ln\prod_{n=1}^\infty\left(1-\frac{1}{4n^2}\right)=\ln\left(\frac{\pi}{2}\right)\\
	\end{align*}
	and the desired formula follows.
\end{proof}

\section{Inequalities for trigonometric sums}\label{sec2}

As we mentioned in the introduction, we are interested in the sum
of cosecants  $I_p$ defined by \eqref{E:I} and the sum of cotangents $J_p$ defined by \eqref{E:J}. Many other trigonometric sums can be expressed in terms of $I_p$ and $J_p$. The next
lemma lists some of these identities.
\bg
\begin{lemma}\label{lm91} For a positive integer $p$ let
	\begin{alignat*}{2}
		K_p&=\sum_{k=1}^{p-1}\tan\left(\frac{k\pi}{2p}\right),
		&\qquad \widetilde{K}_p&=\sum_{k=1}^{p-1}\cot\left(\frac{k\pi}{2p}\right),\\
		L_p&=\sum_{k=1}^{p-1}\frac{k}{\sin(k\pi/p)},
		&\qquad M_p&=\sum_{k=1}^{p}(2k-1)\cot\left(\frac{(2k-1)\pi}{2p}\right)
	\end{alignat*}
	Then,
	
	$\ds
	\begin{matrix}
	\hfill i.&\hfill K_p=&\widetilde{K}_p=I_p.\hfill\label{lm911}\\
	\hfill ii.&\hfill L_p=&(p/2)\,I_p.\hfill\label{lm912}\\
	\hfill iii.&\hfill M_p=&(p/2)\,J_{2p}-2J_p=-p\,I_p.\hfill\label{lm913}\\
	\end{matrix}
	$
\end{lemma}
\begin{proof}
	First, note that the change of summation variable $k\leftarrow p-k$ proves that $K_p=\widetilde{K}_p$. So,
	using the trigonometric identity $\tan\theta+\cot\theta=2\csc(2\theta)$ we obtain $(i)$ as follows:
	\begin{equation*}
		2K_p=K_p+\widetilde{K}_p=\sum_{k=1}^{p-1}\left(\tan\left(\frac{k\pi}{2p}\right)+\cot\left(\frac{k\pi}{2p}\right)\right)
		=2\sum_{k=1}^{p-1}\csc\left(\frac{k\pi}{p}\right)=2I_p
	\end{equation*}

	Similarly, $(ii)$ follows from the change of summation variable $k\leftarrow p-k$ in $L_p$:
	\[
	L_p=\sum_{k=1}^{p-1}\frac{p-k}{\sin(k\pi/p)}=pI_p-L_p
	\]
	Also,
	\begin{align*}
		M_p&=\sum_{\substack{1\leq k<2p\\ k \text{ odd}
			}} k\cot\left(\frac{k\pi}{2p}\right)=\sum_{k=1}^{2p-1} k\cot\left(\frac{k\pi}{2p}\right)- \sum_{\substack{1\leq k<2p\\ k \text{ even}
		}} k\cot\left(\frac{k\pi}{2p}\right)\\
		&=\sum_{k=1}^{2p-1} k\cot\left(\frac{k\pi}{2p}\right)-2 \sum_{k=1}^{p-1} k\cot\left(\frac{k\pi}{p}\right)=J_{2p}-2J_p.
	\end{align*}
	But
	\begin{align*}
		J_{2p}&=\sum_{k=1}^{p-1} k\cot\left(\frac{k\pi}{2p}\right)+\sum_{k=p+1}^{2p-1} k\cot\left(\frac{k\pi}{2p}\right)\\
		&=\sum_{k=1}^{p-1} k\cot\left(\frac{k\pi}{2p}\right)-\sum_{k=1}^{p-1} (2p-k)\cot\left(\frac{k\pi}{2p}\right)\\
		&=2\sum_{k=1}^{p-1} k\cot\left(\frac{k\pi}{2p}\right)-2p\widetilde{K}_p
	\end{align*}
	Thus, using $(i)$ and the trigonometric identity $\cot(\theta/2)-\cot\theta=\csc\theta$ we obtain
	\begin{align*}
		M_p&=J_{2p}-2J_p=2\sum_{k=1}^{p-1} k\left(\cot\left(\frac{k\pi}{2p}\right)-\cot\left(\frac{k\pi}{p}\right)\right)
		-2pI_p\\
		&=2\sum_{k=1}^{p-1}k\csc\left(\frac{k\pi}{p}\right)-2pI_p=2L_p-2pI_p=-pI_p
	\end{align*}
	This concludes the proof of $(iii)$.
\end{proof}

\begin{proposition}\label{pr92}
	For $p\geq2$, let $I_p$ be the sum of cosecants defined by the \eqref{E:I}. Then
	\begin{align*}
		I_p&=-\frac{2\ln 2}{\pi}+\frac{2p}{\pi}E_p,\\
		&=-\frac{2\ln 2}{\pi}+\frac{2p}{\pi}\left(\ln p+\gamma-\ln(\pi/2)\right)+\frac{4p}{\pi}D_p,
	\end{align*}
	where $D_p$ and $E_p$ are defined by formul\ae~ \eqref{E:Dp} and \eqref{E:Ep} respectively.
\end{proposition}
\begin{proof}
	Indeed, our starting point will be the ``simple fractions''  expansion \cite[Chapter 5, \S 2]{ahl} of the cosecant function:
	\[
	\frac{\pi}{\sin(\pi\alpha)}=\sum_{n\in\ent}\frac{(-1)^n}{\alpha-n}=\frac{1}{\alpha}+\sum_{n=1}^\infty(-1)^n\left(\frac{1}{\alpha-n}+
	\frac{1}{\alpha+n}\right)
	\]
	which is valid for $\alpha\in\comp\setminus\ent$. Using this formula with $\alpha=k/p$ for  $k=1,2,\ldots,p-1$ and adding, we conclude that
	\begin{align*}
		\frac{\pi}{p} I_p&=\sum_{k=1}^{p-1}\frac{1}{k}+\sum_{n=1}^\infty(-1)^n\sum_{k=1}^{p-1}\left(\frac{1}{k-np}+
		\frac{1}{k+n p}\right)\\
		&=\sum_{k=1}^{p-1}\frac{1}{k}+ \sum_{n=1}^\infty(-1)^n\left(-\sum_{j=p(n-1)+1}^{pn-1}\frac{1}{j}+
		\sum_{j=pn+1}^{p(n+1)-1}\frac{1}{j}\right),
	\end{align*}
	and this result can be expressed in terms of the Harmonic numbers as follows
	\begin{align*}
		\frac{\pi}{p} I_p&=H_{p-1}+ \sum_{n=1}^\infty(-1)^n\left(- H_{pn-1}+H_{p(n-1)}+H_{p(n+1)-1}-H_{pn}
		\right)\\
		&=H_{p-1}+ \sum_{n=1}^\infty(-1)^n\left(H_{p(n+1)}-2H_{pn}+H_{p(n-1)}\right)
		+\frac{1}{p}\sum_{n=1}^\infty(-1)^n\left(\frac{1}{n}-\frac{1}{n+1}
		\right)\\
		&=H_{p-1}+\sum_{n=1}^\infty(-1)^n\left(H_{p(n+1)}-2H_{pn}+H_{p(n-1)}\right)
		+\frac{1}{p}\left(\sum_{n=1}^\infty\frac{(-1)^n}{n}+\sum_{n=2}^\infty\frac{(-1)^n}{n}\right)\\
		&=H_{p}+\sum_{n=1}^\infty(-1)^n\left(H_{p(n+1)}-2H_{pn}+H_{p(n-1)}\right)
		-\frac{2}{p}\sum_{n=1}^\infty(-1)^{n-1}\frac{1}{n}\\
		&=H_{p}-\frac{2\ln 2}{p}+\sum_{n=1}^\infty(-1)^n\left(H_{p(n+1)}-2H_{pn}+H_{p(n-1)}\right).
	\end{align*}
	Thus
	\begin{align*}
		\frac{\pi}{p} I_p+\frac{2\ln 2}{p}&=
		H_{p}+\sum_{n=1}^\infty(-1)^n\left(H_{p(n+1)}-H_{pn} \right)
		+\sum_{n=1}^\infty(-1)^n\left(H_{p(n-1)}-H_{pn}\right)\\
		&=\sum_{n=0}^\infty(-1)^n\left(H_{p(n+1)}-H_{pn}\right)
		+\sum_{n=1}^\infty(-1)^n\left(H_{p(n-1)}-H_{pn}\right)\\
		&=E_p+E_p=2E_p,
	\end{align*}
	and the desired formula follows according to Lemma~\ref{lm84}.
\end{proof}

 Combining Proposition~\ref{pr92} and Proposition~\ref{pr83}, we obtain:
\begin{proposition}\label{pr93} 
	For $p\geq2$ and $m\geq 1$, we have
	\[
	\pi I_p= 2p\ln p+2(\gamma-\ln( \pi/2))p -\sum_{k=1}^{m-1} \frac{2b_{2k}\eta(2k)}{ k \cdot p^{2k-1}}
	+(-1)^m\frac{2\eta(2m)}{ m\cdot p^{2m-1}}\eps'_{p,m}
	\]
	with $\ds 0<\eps'_{p,m}<\abs{b_{2m}}$.
\end{proposition}
 Using the well-known result (\cite{ wis},\cite[Formula 9.542]{grad}):
 \[
 \eta(2k)=(1-2^{1-2k})\zeta(2k)=\frac{(2^{2k-1}-1)\pi^{2k}(-1)^{k-1}b_{2k}}{(2k)!},
 \]
 and considering separately the cases $m$ even and $m$ odd we obtain
the following corollary.

\begin{corollary}\label{cor94}
	For every positive integer $p$ and every nonnegative integer $n$, the sum of cosecants $I_p$ defined by \eqref{E:I} satisfies the following
	inequalities:
	\begin{align*}
		I_p&<\frac{2p}{\pi}(\ln p+\gamma-\ln(\pi/2))+
		\sum_{k=1}^{2n}(-1)^{k}\frac{(2^{2k}-2)b_{2k}^2}{k\cdot(2k)!}\left(\frac{\pi}{p}\right)^{2k-1},\\
		\noalign{\text{and}}
		I_p&>\frac{2p}{\pi}(\ln p+\gamma-\ln(\pi/2))+
		\sum_{k=1}^{2n+1}(-1)^{k}\frac{(2^{2k}-2)b_{2k}^2}{k\cdot(2k)!}\left(\frac{\pi}{p}\right)^{2k-1}.
	\end{align*}
\end{corollary}
\bg
\qquad As an example, for $n=0$ we obtain the following inequality, valid for every $p\geq1$:
\[
\frac{2p}{\pi}(\ln p+\gamma-\ln(\pi/2))-\frac{\pi}{36p}
<I_p<\frac{2p}{\pi}(\ln p+\gamma-\ln(\pi/2)).
\]
This answers positively the open problem proposed in \cite[Section 7.4]{chen2}.
\bg
\begin{remark}
	The asymptotic expansion of $I_p$ was proposed as an exercise in \cite[Exercise~13, p.~460]{hen}, and 
	it was attributed to P. Waldvogel, but the result there is less precise than Corollary~\ref{cor94} because here we have inequalities valid in the whole range of $p$.
\end{remark}
\bg
 Now we turn our attention to the other trigonometric sum $J_p$.
 The first step
to we find an analogous result to Proposition~\ref{pr92} for the trigonometric sum $J_p$, is the next lemma, where an asymptotic expansion for $J_p$ is proved but it has a harmonic number as an undesired term, later it will be removed.

\begin{lemma}\label{pr72} For every positive integers $p$, there is a real number $\theta_{p}\in(0,1)$ such that
\[
\pi J_p=-p^2H_p+\ln(2\pi) p^2-\frac{p}{2}-\theta_p.
\]
\end{lemma}
\begin{proof}
Indeed, let $\vf$ be the function defined by 
\[\vf(x)=\pi x\cot(\pi x)+\frac{1}{1-x}.\]
According to the partial fraction expansion formula for the cotangent function \cite[Chapter 5, \S 2]{ahl} we know that 
\[
\vf(x)=2+\frac{x}{x+1}+\sum_{n=2}^\infty\left(\frac{x}{x-n}+\frac{x}{x+n}\right).
\]
Thus, $\vf$ is defined and analytic on the interval $(-1,2)$. Let us show that $\vf$ is concave on this interval. Indeed, it is straight forward to check that, for $-1<x<2$ we have
\begin{align*}
\vf^{\prime\prime}(x)&=-\frac{2}{(1+x)^3}-2
\sum_{n=2}^\infty\left(\frac{n}{(n-x)^3}+\frac{n}{(n+x)^3}\right)<0.
\end{align*}
So, we can use Theorem \ref{cor61} with $m=1$ applied to the function $x\mapsto\vf\left(\frac{x+k}{p}\right)$ for $1\le k<p$ to get
\[
0< p\int_{k/p}^{(k+1)/p}\vf(x)dx-\frac{1}{2}\left(\vf\left(\frac{k+1}{p}\right)+\vf\left(\frac{k}{p}\right)\right)\le\frac{3}{2p\pi^2}
\left(\vf'\left(\frac{k}{p}\right)-\vf'\left(\frac{k+1}{p}\right)\right)
\]
Adding these inequalities and noting that $\vf(0)=2$, $\vf'(0)=1$,
$\vf(1)=1$ and $\vf'(1)=-\pi^2/3$, we get
\[
0< p\int_0^1\vf(x)dx-\frac{\pi}{p}J_p-pH_p-\frac{1}{2}\le\frac{3+\pi^2}{2\pi^2p}<\frac{1}{p}
\]
Also, for $x\in[0,1)$, we have
\[
\int_0^x\vf(t)\,dt=-\ln(1-x)+x\ln \sin(\pi x)-\int_0^x \ln\sin(\pi t)\,dt
\]
and, letting $x$ tend to $1$ we obtain
\[
\int_0^1\vf(t)\,dt=\ln(\pi)-\int_0^1\ln\sin(\pi t)\,dt=\ln(2\pi)
\]
where we used the fact   $\int_0^1\ln\sin(\pi t)\,dt=-\ln2$, (see \cite[4.224 Formula 3.]{grad}. So, we have proved that
\[
0< p\ln(2\pi)-\frac{\pi}{p}J_p-pH_p-\frac{1}{2}<\frac{1}{p}
\]
which is equivalent to the desired conclusion.
\end{proof}
\bg
The next proposition gives an analogous result to Proposition~\ref{pr92} for the trigonometric sum $J_p$. 
\bg
\begin{proposition}\label{pr95}
	For a positive integer $p$, 
	let $J_p$ be the sum of cotangents defined by \eqref{E:J}. Then
	\[
	\pi J_p= -p^2\ln p+(\ln(2\pi)-\gamma)p^2 -p+2p^2C_p
	\]
	where $C_p$ is given by \eqref{E:Cp}.
\end{proposition}
\begin{proof}
	Recall that $c_{n}=H_n-\ln n-\gamma-\frac{1}{2n}$ satisfies
	$c_n=\mathcal{O}(1/n^2)$. Thus, both series 
	\[
	C_p=\sum_{n=1}^\infty c_{pn}\quad\text{ and }\quad \widetilde{C}_p=\sum_{n=1}^\infty(-1)^{n-1} c_{pn}
	\] are convergent. Further, we note that $\widetilde{C}_p=D_p-\frac{\ln 2}{2p}$ where $D_p$ is defined by \eqref{E:Dp}.
	
	According to Proposition~\ref{pr92} we have
	\begin{equation}\label{E:pr991}
		\widetilde{C}_p =\frac{\ln(\pi/2)-\gamma-\ln p}{2}+\frac{\pi}{4p}I_p.
	\end{equation}
	Now, noting that
	\begin{align*}
		C_p&=\sum_{\substack{n\geq1\\
				n\,\text{odd}}}c_{pn}
		+\sum_{\substack{n\geq1\\
				n\,\text{even}}}c_{pn}
		=\sum_{\substack{n\geq1\\
				n\,\text{odd}}}c_{pn}+\sum_{n=1}^\infty c_{2pn}\\
		\widetilde{C}_p&=\sum_{\substack{n\geq1\\
				n\,\text{odd}}}c_{pn}
		-\sum_{\substack{n\geq1\\
				n\,\text{even}}}c_{pn}
		=\sum_{\substack{n\geq1\\
				n\,\text{odd}}}c_{pn}-\sum_{n=1}^\infty c_{2pn}
	\end{align*}
	we conclude that $C_p-\widetilde{C}_p=2C_{2p}$, or equivalently
	\begin{equation}\label{E:pr992}
		C_p-2C_{2p}=\widetilde{C}_p
	\end{equation}
	On the other hand, for a positive integer $p$ let us define $F_p$ by
	\begin{equation}\label{E:Fp}
		F_p=\frac{\ln p+\gamma-\ln(2\pi)}{2}+\frac{1}{2p}+\frac{\pi}{2p^2}J_p.
	\end{equation}
	It is easy to check, using Lemma~\ref{lm91} $(iii)$, that
	\begin{align}\label{E:pr994}
		F_p-2F_{2p}&=\frac{\ln(\pi/2)-\ln p-\gamma}{2} -\frac{\pi}{4p^2}(J_{2p}-2J_p)\notag\\
		&=\frac{\ln(\pi/2)-\ln p-\gamma}{2} +\frac{\pi}{4p}I_p
	\end{align}
	We conclude from \eqref{E:pr992} and \eqref{E:pr994} that $C_p-2C_{2p}=F_p-2F_{2p}$, or equivalently
	\[C_p-F_p=2(C_{2p}-F_{2p}).\]
	Hence, 
	\begin{equation}\label{E:pr995}
		\forall\,m\geq1,\qquad C_p-F_p=2^m(C_{2^mp}-F_{2^mp})
	\end{equation}
	Now, using  Lemma~\ref{pr81} to replace $H_p$ in Lemma~\ref{pr72},  we obtain
	\begin{align*}
		\frac{\pi}{p^2}J_p&=\ln(2\pi)-H_p-\frac{1}{2p}+\mathcal{O}\left(\frac{1}{p^2}\right)\\
		&=\ln(2\pi)-\ln p-\gamma-\frac{1}{p} +\mathcal{O}\left(\frac{1}{p^2}\right)
	\end{align*}
	Thus $F_p=\mathcal{O}\left(\frac{1}{p^2}\right)$. Similarly, from the fact that $c_n=\mathcal{O}\left(\frac{1}{n^2}\right)$
	we conclude also that $C_p=\mathcal{O}\left(\frac{1}{p^2}\right)$. Consequently, there exists a constant $\kappa$  such that, for large values of $p$
	we have $\abs{C_p-F_p}\leq \kappa/p^2$. So, from \eqref{E:pr995}, we see that for large values of $m$ we have
	\[
	\abs{C_p-F_p}\leq\frac{\kappa}{2^mp^2}
	\]
	and letting $m$ tend to $+\infty$ we obtain $C_p=F_p$, which is equivalent to the announced result.
\end{proof}
\bg
\qquad Combining Proposition~\ref{pr95} and Proposition~\ref{pr82}, we obtain:
\begin{proposition}\label{pr96} 
	For $p\geq2$ and $m\geq 1$, we have
	\[
	\pi J_p= -p^2\ln p+(\ln(2\pi)-\gamma)p^2 -p -\sum_{k=1}^{m-1}\frac{b_{2k}\zeta(2k)}{ k\cdot p^{2k-2}}
	+(-1)^m\frac{\zeta(2m)}{m\cdot p^{2m-2}}\eps_{p,m},
	\]
	with $0<\eps_{p,m}<\abs{b_{2m}}$, where $\zeta$ is the well-known Riemann zeta function.
\end{proposition}
\qquad Using the values of the $\zeta(2k)$'s \cite[Formula 9.542]{grad}), and considering separately
the cases $m$ even and $m$ odd we obtain
the next corollary.

\begin{corollary}\label{cor97}
	For every positive integer $p$ and every nonnegative integer $n$, the sum of cotangents $J_p$ defined by \eqref{E:J} satisfies the following
	inequalities:
	\begin{align*}
		J_p&< \frac{1}{\pi}\left(-p^2\ln p+(\ln(2\pi)-\gamma)p^2 -p\right) 
		+2\pi\sum_{k=1}^{2n}(-1)^k\frac{b_{2k}^2}{ k\cdot(2k)!} \left(\frac{2\pi}{ p}\right)^{2k-2},\\
		\noalign{\text{and}}
		J_p&> \frac{1}{\pi}\left(-p^2\ln p+(\ln(2\pi)-\gamma)p^2 -p\right) 
		+2\pi\sum_{k=1}^{2n+1}(-1)^k\frac{b_{2k}^2}{ k\cdot(2k)!} \left(\frac{2\pi}{ p}\right)^{2k-2}.
	\end{align*}
\end{corollary}
\bg
 As an example, for $n=0$ we obtain the following double inequality, which is valid for  $p\geq1$ :
\[
0 < \frac{1}{\pi}\left(-p^2\ln p+(\ln(2\pi)-\gamma)p^2 -p\right)-J_p<\frac{\pi}{36}
\]
\bg

\begin{remark} \label{rm98} Note that we have proved the following results. For a postive integer $p$:
	\begin{align*}
		\sum_{n=1}^\infty(-1)^{n-1}(H_{pn}-\ln(pn)-\gamma)&=\frac{\ln(\pi/2)-\gamma-\ln p}{2}+\frac{\ln 2}{2p}
		+\frac{\pi}{4p}\sum_{k=1}^{p-1}\csc\left(\frac{k \pi}{p}\right).\\
		\sum_{n=0}^\infty(-1)^{n}(H_{p(n+1)}-H_{pn})&=\frac{\ln 2}{p}
		+\frac{\pi}{2p}\sum_{k=1}^{p-1}\csc\left(\frac{k \pi}{p}\right).\\
		\sum_{n=1}^\infty\left(H_{pn}-\ln(pn)-\gamma-\frac{1}{2pn}\right)&=
		\frac{\ln p+\gamma-\ln(2\pi)}{2}+\frac{1}{2p}+\frac{\pi}{2p^2}\sum_{k=1}^{p-1}k\cot\left(\frac{k \pi}{p}\right).
	\end{align*}
	These results are to be compared with those in \cite{kou}, see also \cite{kou1}.
\end{remark}


\end{document}